\documentclass[11pt]{article}
\usepackage{amsmath}
\usepackage{amsthm}
\usepackage{amssymb}
\usepackage{algorithm}
\usepackage{subfig}
\usepackage{color}
\usepackage[english]{babel}
\usepackage{graphicx}
\usepackage{wrapfig,epsfig}
\usepackage{epstopdf}
\usepackage[hyphens]{url} 
\usepackage{graphicx}
\usepackage{color}
\usepackage{epstopdf}
\usepackage{algpseudocode}
\usepackage{scrextend}
\usepackage[T1]{fontenc}
\usepackage{bbm}
\usepackage{comment}

\let\C\relax
\usepackage{tikz}
\usepackage{hyperref}  
\hypersetup{colorlinks=true,citecolor=blue,linkcolor=blue} 
\usetikzlibrary{arrows}
\usepackage[margin=1in]{geometry}
\graphicspath{{./figs/}}

\author{
  Rasmus Kyng\thanks{
      Supported by ONR grant
      N00014-18-1-2562.} \\ 
  \texttt{kyng@inf.ethz.ch}\\
  ETH Zurich
  \and
  Kyle Luh\thanks{Supported in part by NSF postdoctoral fellowship DMS-1702533.} \\
  \texttt{kyle.luh@colorado.edu}\\
  University of Colorado, Boulder
  \and
  Zhao Song\\
  \texttt{zhaos@princeton.edu}\\
  Princeton University 
}

\date{\today}

\title{Four Deviations Suffice for Rank 1 Matrices}
\newtheorem{theorem}{Theorem}[section]
\newtheorem{lemma}[theorem]{Lemma}
\newtheorem{definition}[theorem]{Definition}

\newtheorem{proposition}[theorem]{Proposition}
\newtheorem{corollary}[theorem]{Corollary}

\newtheorem{fact}[theorem]{Fact}
\newtheorem{remark}[theorem]{Remark}

\newcommand{\eps}{\varepsilon}

\newcommand{\R}{\mathbb{R}}

\renewcommand{\P}{\mathbb{P}}
\renewcommand{\Im}{\operatorname{Im}}
\newcommand{\T}{*}
\newcommand{\Ab}{\mathbf{Ab}}

\DeclareMathOperator*{\E}{{\mathbb{E}}}
\DeclareMathOperator*{\Var}{{\bf {Var}}}
\DeclareMathOperator*{\C}{\mathbb{C}}

\DeclareMathOperator{\Tr}{Tr}

\newcommand{\setof}[1]{\left\{#1  \right\}}
\newcommand{\abs}[1]{\lvert #1  \rvert}

\newcommand{\todolow}[1]{\textbf{\color{yellow}[TODO: #1]}}
\renewcommand{\todolow}[1]{}

\makeatletter
\newcommand*{\RN}[1]{\expandafter\@slowromancap\romantaumeral #1@}
\makeatother

\usepackage{lineno}

\usepackage{etoolbox}

\makeatletter

\newcommand{\define}[4][ignore]{%
  \ifstrequal{#1}{ignore}{}{
  \@namedef{thmtitle@#2}{#1}}%
  \@namedef{thm@#2}{#4}%
  \@namedef{thmtypen@#2}{lemma}%
  \newtheorem{thmtype@#2}[theorem]{#3}%
  \newtheorem*{thmtypealt@#2}{#3~\ref{#2}}%
}

\newcommand{\state}[1]{%
  \@namedef{curthm}{#1}
  \@ifundefined{thmtitle@#1}{
  \begin{thmtype@#1}
    }{
  \begin{thmtype@#1}[\@nameuse{thmtitle@#1}]
  }
    \label{#1}
    \@nameuse{thm@#1}
  \end{thmtype@#1}
  \@ifundefined{thmdone@#1}{
  \@namedef{thmdone@#1}{stated}%
  }{}
}

\newcommand{\restate}[1]{%
  \@namedef{curthm}{#1}
  \@ifundefined{thmtitle@#1}{
    \begin{thmtypealt@#1}
    }{
  \begin{thmtypealt@#1}[\@nameuse{thmtitle@#1}]
  }
    \@nameuse{thm@#1}
  \end{thmtypealt@#1}
  \@ifundefined{thmdone@#1}{
  \@namedef{thmdone@#1}{stated}%
  }{}
}

\newcommand{\thmlabel}[1]{
  \@ifundefined{thmdone@\@nameuse{curthm}}{\label{#1}
    }{\tag*{\eqref{#1}}}
}
\makeatother

\newcommand\blfootnote[1]{%
	\begingroup
	\renewcommand\thefootnote{}\footnote{#1}%
	\addtocounter{footnote}{-1}%
	\endgroup
}

\begin{document}

  \maketitle
  \begin{abstract}We prove a matrix discrepancy bound that strengthens
    the famous Kadison-Singer result of Marcus, Spielman, and Srivastava.
Consider any independent scalar random variables $\xi_1, \ldots, \xi_n$ with finite support, e.g.
 $\{ \pm 1 \}$ or  $\{ 0,1 \}$-valued random variables, or some combination thereof.
Let $\mathbf{u}_1, \dots, \mathbf{u}_n \in \C^m$ and
$$
\sigma^2 = \left\| \sum_{i=1}^n \Var[ \xi_i ] (\mathbf{u}_i \mathbf{u}_i^{*})^2 \right\|.
$$
Then there exists a choice of outcomes  $\eps_1,\ldots,\eps_n$
in the support of $\xi_1,
\ldots, \xi_n$
s.t. 
$$
\left \|\sum_{i=1}^n \E [ \xi_i] \mathbf{u}_i \mathbf{u}_i^*  - \sum_{i=1}^n \eps_i \mathbf{u}_i \mathbf{u}_i^* \right \| \leq 4 \sigma.
$$ 
A simple consequence of our result is an improvement of a Lyapunov-type theorem of Akemann and Weaver.

  \end{abstract}
  \thispagestyle{empty}




\section{Introduction}
\blfootnote{{\it 2010 Mathematics Subject Classification}.
	Primary: 05A99, 11K38 ; Secondary: 46L05.}
\blfootnote{{\it Key words and phrases}. Matrix discrepancy, Interlacing polynomials, Lyapunov Theorem, Operator algebra.}	
Discrepancy theory lies at the heart of numerous problems in mathematics and computer science \cite{Chazelle2000}.  Although closely tied to probability theory, direct randomized approaches rarely yield the best bounds.  In a classical formulation in discrepancy theory, we have $n$ sets on $n$ elements, and would like to two-color
the elements so that each set has roughly the same number of elements
of each color.  Using a simple random coloring, it is an easy
consequence of Chernoff's bound that there exists a coloring such that
the discrepancy in all $n$ sets is $O(\sqrt{n \log n})$
\cite{AlonSpencer}.  However, in a celebrated result, Spencer showed
that in fact there is a coloring with discrepancy at most $6 \sqrt{n}$
\cite{Spencer85}. 

Recently, there has been significant success in generalizing
Chernoff/Hoeffding/Bernstein/Bennett-type concentration bounds for
scalar random variables to matrix-valued random variables
\cite{Rudelson99, aw02, Tropp12}.    
Consider the following matrix concentration bound, which is a direct
consequence of a matrix Hoeffding bound.
\begin{theorem}[\cite{Tropp12}]\label{thm:absTropp}
Let $\xi_i \in \{\pm 1\}$ be independent, symmetric random signs and
$A_1, \dots, A_n \in \C^{m\times m}$ be positive semi-definite matrices.
Suppose $\max_i \| A_i \| \leq \epsilon$ and $\left\| \sum_{i=1}^n A_i
\right\| \leq 1$.
  Then,
$$
\P \left( \left  \|\sum_{i=1}^n \xi_i A_i \right \| \geq t \sqrt{\epsilon} \right) \leq 2m \exp(-t^2/2).
$$
\end{theorem}
A consequence of this theorem is that with high probability

\begin{align}
\label{eq:highprobmatnormconc}
\left \| \sum_{i=1}^n \xi_i A_i \right \| = O(\sqrt{\log m})
  \sqrt{\epsilon}.
\end{align}
The main theorem of \cite{MSS15b} (from which they deduce the Kadison-Singer Theorem) is essentially equivalent
to the following statement (which can readily be derived from the
bipartition statement in \cite{MSS15b}).

\begin{theorem}[\cite{MSS15b}]\label{thm:KS}
Let $\mathbf{u}_1, \dots, \mathbf{u}_n \in \C^m$ and
suppose $\max_{i \in [n]} \| \mathbf{u}_i \mathbf{u}_i^*\| \leq \epsilon$ and $ \sum_{i=1}^n  u_i u_i^*
  = I$.
  Then, there exists signs $\xi_i \in \{\pm 1\}$ s.t.
$$
 \left \|\sum_{i=1}^n \eps_i \mathbf{u}_i \mathbf{u}_i^*\right \| \leq  O( \sqrt{\epsilon}) .
$$
\end{theorem}
Thus, for rank 1 matrices the theorem improves on the norm bound in
Equation~\eqref{eq:highprobmatnormconc}, by a factor $\sqrt{\log m}$,
in a manner analogous to the improvement of Spencer's theorem over the bound based
on the scalar Chernoff bound.

For random signings of matrices one can establish bounds in some cases that are much stronger
than Theorem~\ref{thm:absTropp}.
\begin{theorem}[\cite{Tropp12}] \label{thm:Tropp}
Let $\xi_i \in \{\pm 1\}$ be independent random signs, and let
$A_1, \dots, A_n \in \C^{m\times m}$ be Hermitian matrices.
Let $\sigma^2 = \left\| \sum_{i=1}^n \Var[ \xi_i ] A_i ^2 \right\|.$
 Then, 
$$
\P \left( \left  \|\sum_{i=1}^n \E [ \xi_i] A_i - \sum_{i=1}^n \xi_i A_i \right \| \geq t \cdot \sigma \right) \leq 2m \exp(-t^2/2).
$$
\end{theorem}
From this theorem we deduce that with high probability
\begin{equation}
\label{eq:highprobmatnormconcVAR}
 \left \|\sum_{i=1}^n \E [ \xi_i] A_i - \sum_{i=1}^n \xi_i A_i \right \|  = O(\sqrt{\log m})
\sigma.
\end{equation}
Of course, this implies that there exists a choice of signs $\eps_1, \dots, \eps_n \in \{\pm 1\}$ such that 
$$
\left \|\sum_{i=1}^n \E [ \xi_i] A_i - \sum_{i=1}^n \eps_i A_i \right \|  = O(\sqrt{\log m})
\sigma.
$$
Our main result demonstrates that for rank-1 matrices, there exists a choice of signs with a stronger guarantee.  
\begin{theorem}[Main Theorem] \label{thm:mainexpct}
Consider any independent scalar random variables $\xi_1, \ldots,
\xi_n$ with finite support.
Let $\mathbf{u}_1, \dots, \mathbf{u}_n \in \C^m$ and
$$
\sigma^2 = \left\| \sum_{i=1}^n \Var[ \xi_i ] (u_i u_i^{*})^2 \right\|.
$$

Then there exists a choice of outcomes  $\eps_1,\ldots,\eps_n$
in the support of $\xi_1,
\ldots, \xi_n$
$$
\left \|\sum_{i=1}^n \E [ \xi_i] \mathbf{u}_i \mathbf{u}_i^*  - \sum_{i=1}^n \eps_i \mathbf{u}_i \mathbf{u}_i^* \right \| \leq 4 \sigma.
$$ 
\end{theorem}

For rank 1 matrices our theorem improves on the norm bound in
Equation~\eqref{eq:highprobmatnormconcVAR}, by a factor $\sqrt{\log m}$.
Note for example, if  $\xi_i$ is $\setof{\pm 1}$-valued,
then $\Var[ \xi_i ]= 1-\E [ \xi_i]^2 $.

\begin{remark}
	As will be apparent in the proof, given the support of $\xi_i$, one can consider alternative random variables $\xi_i'$ with identical mean and reduced variance such that the support of $\xi_i'$ is a subset of $\xi_i$, to obtain the same conclusion.  In particular, in some applications, one may be interested in restricting the values taken by $\eps_i$.  For example, shifting the probability mass of $\xi_i$ to two points of support, the nearest above the mean and the closest below it, while maintaining the mean can reduce the variance and therefore guarantee that our choices of $\eps_i$ lie on these two points.  
\end{remark}
Specializing to centered random variables, we obtain the following
corollary, which in a simple way generalizes Theorem \ref{thm:KS}, although
with a slightly worse constant.
\begin{corollary} \label{thm:main}
Let $\mathbf{u}_1, \dots, \mathbf{u}_n \in \C^m$ and
$$
\sigma^2 = \left\| \sum_{i=1}^n (\mathbf{u}_i \mathbf{u}_i^{*})^2 \right\|.
$$
There exists a choice of signs $\eps_i \in \setof{\pm 1}$ such that
$$
\left \| \sum_{i=1}^n \eps_i \mathbf{u}_i \mathbf{u}_i^* \right \| \leq 4 \sigma.
$$ 
\end{corollary}
Notice that if $\max_i \| \mathbf{u}_i \mathbf{u}_i^*\| \leq \epsilon$ and $ \sum_{i=1}^n  \mathbf{u}_i \mathbf{u}_i^*
  = I$, then $\sigma^2 \leq \epsilon$, and we obtain
  Theorem~\ref{thm:KS} as consequence of the corollary.

Our Theorem~\ref{thm:mainexpct} has multiple advantages over
Theorem~\ref{thm:KS}.
It allows us to show existence of solutions close to the mean of arbitrarily
biased $\pm 1$ random variables, instead of only zero mean distributions.
If some variable $\xi_i$ is extremely biased, its
variance is correspondingly low, as $\Var[ \xi_i ] = 1-\E [ \xi_i]^2$.

If only a small subset of the rank 1 matrices obtain the $\epsilon$ norm
bound, and the rest are significantly smaller in norm, then we can have $\sigma^2 \approx \epsilon^2$, so
we prove a bound of $O(\epsilon)$ instead of the $O(\sqrt{\epsilon})$
bound of the Kadison-Singer theorem.
On the other hand, when the problem is appropriately scaled, we always have $\sigma \geq \epsilon^2$, so the gap
between the two results can never be more than a square root.  Additionally, note that although the Kadison-Singer theorem requires
$ \sum_{i=1}^n  \mathbf{u}_i \mathbf{u}_i^*
  = I$, a multi-paving argument\footnote{This was pointed out to
    us by Tarun Kathuria.} can be used to instead 
  relax this to $\| \sum_{i=1}^n  \mathbf{u}_i \mathbf{u}_i^* \| \leq 1$ \cite[Theorem 2 (b)]{Weaver04}.

\paragraph{Approximate Lyapunov Theorems.}
Marcus, Spielman and Srivastava resolved the Kadison-Singer problem by proving Weaver's conjecture \cite{Weaver04}, which was shown to imply the Kadison-Singer conjecture.  In \cite{Akemann14}, Akemann and Weaver prove a generalization of Weaver's conjecture \cite{Weaver04}.
\begin{theorem}[\cite{Akemann14}] \label{thm:lyapunovthm}
Let $\mathbf{u}_1, \dots, \mathbf{u}_n \in \C^m$ such that $\|\sum_{i=1}^n \mathbf{u}_i \mathbf{u}_i^* \| \leq 1$ and $\max_i \| \mathbf{u}_i \mathbf{u}_i^*\| \leq \epsilon$.  For any $t_i \in [0,1]$ and $1 \leq i \leq n$, there exists a set of indices $S \subset \{1, 2, \dots, n\}$ such that
$$
\left \| \sum_{i \in S} \mathbf{u}_i \mathbf{u}_i^* - \sum_{i=1}^n t_i \mathbf{u}_i \mathbf{u}_i^* \right \| = O(\epsilon^{1/8}).
$$ 
\end{theorem}

Due to the classical Lyapunov theorem \cite{lyapunov1940completely} and its equivalent versions \cite{Lindenstrauss66}, in their study of operator algebras, Akemann and Anderson \cite{Ademann91} refer to a result as a Lyapunov theorem if the result states that for a convex set $C$, the image of $C$ under an affine map is equal to the image of the extreme points of $C$.  Theorem \ref{thm:lyapunovthm} is an approximate Lyapunov theorem as it can be interpreted as saying that the image of $[0,1]^n$ under the map
$$
f: (t_1, \dots, t_n) \rightarrow \sum_{i=1}^n t_i \mathbf{u}_i \mathbf{u}_i^*, 
$$
can be approximated by the image of one of the vertices of the hypercube $[0,1]^n$.  A corollary of our main result, Theorem \ref{thm:mainexpct}, is the following strengthening of Theorem \ref{thm:lyapunovthm}.  This result greatly improves the $\epsilon$ dependence of the original Lyapunov-type theorem and provides a small explicit constant.
\begin{corollary} \label{cor:lyapunov}
Let $\mathbf{u}_1, \dots, \mathbf{u}_n \in \C^m$ such that $\|\sum_{i=1}^n \mathbf{u}_i \mathbf{u}_i^* \| \leq 1$ and $\max_i \| \mathbf{u}_i \mathbf{u}_i^*\| \leq \epsilon$.  For any $t_i \in [0,1]$ and $1 \leq i \leq n$, there exists a set of indices $S \subset \{1, 2, \dots, n\}$ such that
$$
\left \| \sum_{i \in S} \mathbf{u}_i \mathbf{u}_i^* - \sum_{i=1}^n t_i \mathbf{u}_i \mathbf{u}_i^* \right \| \leq 2 \, \epsilon^{1/2}.
$$ 
\end{corollary}
The corollary follows immediately from Theorem~\ref{thm:mainexpct} by
choosing as the $\xi_i$ a set of independent
 $\setof{0,1}$-valued random variables with means $t_i \in [0,1]$.
Note then that $\Var[ \xi_i ] = t_i (1-t_i) \leq 1/4$, and so by the assumptions
$\max_i \| \mathbf{u}_i \mathbf{u}_i^*\| \leq \epsilon$ and $ \sum_{i=1}^n  \mathbf{u}_i \mathbf{u}_i^*
  \preceq I$, we have $\sigma^2 \leq \epsilon/4 $, and we obtain
  Theorem~\ref{thm:KS} as consequence of the corollary.



\subsection{Our techniques}
Our proof is based on the method of interlacing polynomials introduced
in \cite{MSS15a, MSS15b}.  
One difficulty in applying the method of interlacing polynomials is the
inability to control both the largest and smallest eigenvalues
of a matrix simultaneously.
Various techniques and restrictions are used to overcome this problem
in \cite{MSS15a, MSS15b} (studying bipartite graphs, assuming
isotropic position).
We develop a seemingly more natural approach for simultaneously controlling both the
largest and smallest roots of the matrices we consider.
We study polynomials that can be viewed as expected characteristic
polynomials, but are more easily understood as the expectation of a product of multiple
determinants.
The using of a product of two determinants helps us bound the upper and lower eigenvalues
both at the same time.

We introduce an analytic expression for the expected
polynomials in terms of linear operators that use second order derivatives.
This has the advantage of allowing us to gain stronger control over
the movement of roots of polynomials under the linear operators we
apply than those used by \cite{MSS15b}.
This is because movement of the roots now depends on the curvature of our polynomials in a
favorable way. 
Interestingly, linear operators containing second order derivatives also
appear in the work of \cite{ag14} but for different reasons.
This lets us reuse one of their lemmas for bounding the positions of the
roots of a real stable polynomial.





\section*{Acknowledgements}  The authors would like to thank the anonymous referees for carefully reading the manuscript and providing helpful comments that improved the presentation of this work.
\section{Preliminaries}
We gather several basic linear algebraic and analytic facts in the following sections.
\subsection{Linear Algebra}
\begin{lemma} \label{lem:rankone}
Let $\mathbf{x} \in \R^n$.  Then
$$
\det(I - t\mathbf{x} \mathbf{x}^{\T}) = 1 - t \mathbf{x}^{\T} \mathbf{x} .
$$
\end{lemma}

\begin{fact}[Jacobi's Formula] \label{fact:Jacobi}
For $A(t) \in \R^{n \times n}$ a function of $t$, 
$$
\frac{d}{dt} \det(A(t)) =  \det(A(t)) \Tr \Big[A^{-1}(t) \frac{d}{dt} A(t)\Big] .
$$
\end{fact}

\subsection{Real Stability}
\begin{definition}
A multivariate polynomial $p(z_1, \dots, z_n) \in \C[z_1, \dots, z_n]$ is \emph{stable} if it has no zeros in the region $\{(z_1, \dots, z_n) : \Im(z_i) > 0 \text{ for all } 1 \leq i \leq n \}$.  $p$ is \emph{real stable} if $p$ is stable and the coefficients of $p$ are real.  
\end{definition}

\begin{lemma}[Corollary 2.8, \cite{Anari18}] \label{lem:derivativestable}
If $p \in \R[z_1, \dots, z_n]$ is real stable, then for any $c > 0$, so is
$$
(1 -c \partial_{z_i}^2) p(z_1, \dots, z_n)
$$
for all $1 \leq i \leq n$.

\end{lemma}

\begin{lemma}[Proposition 2.4, \cite{bb08}] \label{lem:determinantstable}
If $A_1, \dots, A_n$ are positive semidefinite symmetric matrices, then the polynomial
$$
\det \left( \sum_{i=1}^n z_i A_i \right)
$$
is real stable.  
\end{lemma}

We also need that real stability is preserved under fixing variables
to real values (see \cite[Lemma 2.4(d)]{wagner}).
\begin{proposition} \label{prop:setreal}
If $p\in\R[z_1,\ldots,z_m]$ is real stable and $a\in \R$, then
$p|_{z_1=a}=p(a,z_2,\ldots,z_m)\in\R[z_2,\ldots,z_m]$ is real stable.
\end{proposition}


\subsection{Interlacing Families}
We recall the definition and consequences of interlacing families from \cite{MSS15a}.
\begin{definition} \label{def:interlacing}
We say a real rooted polynomial $g(x) = C \prod_{i=1}^{n-1} (x - \alpha_i)$ \emph{interlaces} the real rooted polynomial $f(x) = C' \prod_{i=1}^n (x - \beta_i)$ if 
$$
\beta_1 \leq \alpha_1 \leq \dots \leq \alpha_{n-1} \leq \beta_n.
$$
Polynomials $f_1, \dots, f_k$ have a \emph{common interlacing} if there is a polynomial $g$ that interlaces each of the $f_i$.
\end{definition}
The following lemma relates the roots of a sum of polynomials to those of a common interlacer. 
\begin{lemma}[Lemma 4.2, \cite{MSS15a}] \label{lem:rootupperbound}
Let $f_1, \dots, f_k$ be degree $d$ real rooted polynomials with positive leading coefficients.  Define 
$$
f_{\emptyset} := \sum_{i=1}^k f_i.
$$
If $f_1, \dots, f_k$ have a common interlacing then there exists an $i$ for which the largest root of $f_i$ is upper bounded by the largest root of $f_0$.
\end{lemma}

\begin{definition} \label{def:interlacingfamily}
Let $S_1, \dots, S_n$ be finite sets.  For each choice of assignment $s_1, \dots, s_n \in S_1 \times \dots \times S_n$, let $f_{s_1, \dots, s_n}(x)$ be a real rooted degree $d$ polynomial with positive leading coefficient.  For a partial assignment $s_1, \dots, s_k \in S_1 \times \dots \times S_k$ for $k < n$, we define
\begin{equation} \label{eq:conditionalpoly}
f_{s_1, \dots, s_k } := \sum_{s_{k+1} \in S_{k+1}, \dots, s_n \in S_n} f_{s_1, \dots, s_k, s_{k+1}, \dots, s_n}.
\end{equation}
Note that this is compatible with our definition of $f_{\emptyset}$ from Lemma \ref{lem:rootupperbound}.  We say that the polynomials $\{f_{s_1, \dots, s_n} \}$ form an \emph{interlacing family} if for any $k = 0, \dots, n-1$ and all $s_1, \dots, s_k \in S_1 \times \dots \times S_k$, the polynomials have a common interlacing. 
\end{definition}
The following lemma relates the roots of the interlacing family to those of $f_{\emptyset}$. 
\begin{lemma}[Theorem 4.4, \cite{MSS15a}] \label{lem:rootbound}
Let $S_1, \dots, S_n$ be finite sets and let $\{f_{s_1, \dots, s_n}\}$ be an interlacing family. Then there exists some $s_1, \dots, s_n \in S_1 \times \dots \times S_n$ so that the largest root of $f_{s_1, \dots, s_n}$ is upper bounded by the largest root of $f_{\emptyset}$.
\end{lemma}

Finally, we recall a relationship between real-rootedness and common interlacings which has been discovered independently several times \cite{Dedieu94, Fell80, CS07}.
\begin{lemma} \label{lem:commoninterlacing}
Let $f_1, \dots, f_k$ be univariate polynomials of the same degree with positive leading coefficient.  Then $f_1, \dots, f_k$ have a common interlacing if and only if $\sum_{i=1}^k \alpha_i f_i$ is real rooted for all nonnegative $\alpha_i$ such that $\sum_i \alpha_i = 1$.
\end{lemma}

\section{Expected Characteristic Polynomial}
Instead of working with the random polynomial $\det(x I - \sum_{i=1}^n \xi_i \mathbf{u}_i \mathbf{u}_i^{\T})$, we consider
$$
\det\left(x^2 I - \Big(\sum_{i=1}^n \xi_i \mathbf{u}_i \mathbf{u}_i^{\T}\Big)^2 \right) = \det\left( x I - \sum_{i=1}^n \xi_i \mathbf{u}_i \mathbf{u}_i^{\T} \right) \det\left( x I + \sum_{i=1}^n \xi_i \mathbf{u}_i \mathbf{u}_i^{\T} \right).
$$
Observe that the largest root $\lambda_{\max}$
of this polynomial is
\begin{equation}
  \label{eq:rootisnorm}
  \lambda_{\max} \left( 
\det\left(x^2 I - \Big(\sum_{i=1}^n \xi_i \mathbf{u}_i \mathbf{u}_i^{\T}\Big)^2 \right)
\right)
=
\left\| \sum_{i=1}^n \xi_i \mathbf{u}_i \mathbf{u}_i^{\T} \right\|
.
\end{equation}

We gather some results that will allow us to extract an analytic expression for the expected characteristic polynomial.
\begin{lemma} \label{lem:expectedchar}
For positive semidefinite (PSD) matrices 
$M, N \in \R^{m \times m}$, $v \in \R^m$ and $\xi$ a random variable with zero mean and variance $\tau^2$,
\begin{equation} \label{eq:expectedchar}
\E_{\xi} \left[ \det(M - \xi \mathbf{v} \mathbf{v}^{\T}) \det(N + \xi \mathbf{v} \mathbf{v}^{\T}) \right] = \left(1 -
  \frac{1}{2} \frac{d^2}{dt^2} \right) \Bigg |_{t=0} \det(M + t
\tau \mathbf{v} \mathbf{v}^{\T}) \det(N + t \tau \mathbf{v} \mathbf{v}^{\T}).
\end{equation}
\end{lemma}
\begin{proof}
We can assume that both $M$ and $N$ are positive definitive and hence
invertible. The argument for the the positive semi-definite case 
follows by a continuity argument (using Hurwitz's theorem from complex
analysis, see also \cite{MSS15b}). 

We show that the two sides of \eqref{eq:expectedchar} are equivalent to the same expression.  Beginning on the left hand side, 
\begin{align*}
\E[\det(M - \xi \mathbf{v} \mathbf{v}^{\T}) &\det(N + \xi \mathbf{v} \mathbf{v}^{\T})] \\
&= \det(M) \det(N) \E[\det(I - \xi M^{-1/2} \mathbf{v} \mathbf{v}^{\T} M^{-1/2}) \det(I + \xi N^{-1/2} \mathbf{v} \mathbf{v}^{\T} N^{-1/2})] \\
&=  \det(M) \det(N) \E[1 + \xi \mathbf{b}^{\T} \mathbf{b} - \xi \mathbf{a}^{\T}\mathbf{a} - \xi^2 \mathbf{a}^{\T} \mathbf{a}  \mathbf{b}^{\T} \mathbf{b}] \\
&= \det(M) \det(N) (1 - \tau^2 \mathbf{a}^{\T} \mathbf{a} \mathbf{b}^{\T} \mathbf{b})
\end{align*}
where $\mathbf{a} := M^{-1/2} \mathbf{v}$ and $\mathbf{b} := N^{-1/2} \mathbf{v}$. 
For the right hand side of (\ref{eq:expectedchar}), 
\begin{align*}
\left(1 - \frac{1}{2} \frac{d^2}{dt^2}  \right) \Bigg |_{t=0} &\det(M                                                                
+ t \tau \mathbf{v} \mathbf{v}^{\T}) \det(N + t \tau \mathbf{v} \mathbf{v}^{\T})  \\
&= \det(M) \det(N) \left(1 - \frac{1}{2} \frac{d^2}{dt^2} \right)
  \Bigg |_{t=0} \det(I + t \tau 
  \mathbf{a} \mathbf{a}^{\T}) \det(I + t \tau \mathbf{b} \mathbf{b}^{\T}) \\
&= \det(M) \det(N) (1- \tau^2 \mathbf{a}^{\T} \mathbf{a} \mathbf{b}^{\T} \mathbf{b})
\end{align*}\
where the last line follows from Lemma \ref{lem:rankone}.
\end{proof}

Centering our random variables and applying the previous lemma leads to the following corollary for non-centered random variables.
\begin{corollary} \label{cor:univariateidentity}
Let $M, N \in \R^{m \times m}$ be arbitrary PSD matrices, $\mathbf{v} \in \R^m$
and $\xi$ a random variable with expectation $\mu$ and variance
$\tau^2$. 
$$
\E_{\xi} [\det(M - (\xi - \mu) \mathbf{v} \mathbf{v}^{\T} ) \det(N + (\xi-\mu) \mathbf{v} \mathbf{v}^{\T})] = \left( 1 -
  \frac{1}{2} \frac{d^2}{dt^2} \right) \Bigg |_{t=0} \det(M  + t \tau \mathbf{v} \mathbf{v}^{\T}) \det(N + t \tau \mathbf{v} \mathbf{v}^{\T}).
$$
\end{corollary}

We can now derive an expression for the expected characteristic polynomial.
\begin{proposition} \label{prop:expectedcharpoly}
Let $\mathbf{u}_1, \dots, \mathbf{u}_n \in \R^m$.   Consider independent random variables $\xi_{i}$ with means $\mu_{i}$ and variances $\tau_{i}^2$.  Let $Q \in \R^{m \times m}$ be a symmetric matrix.  
\begin{align*}
 \E_{\xi} \left[ \det\left(x^2 I - \Big(Q + \sum_{i = 1}^n (\xi_i - \mu_i) \mathbf{u}_i \mathbf{u}_i^T
     \Big)^2 \right) \right]  
= \prod_{i=1}^n \left(1 -\frac{\partial_{z_i}^2}{2} \right)  \Bigg|_{z_i = 0} 
& ~~~ \det\Big(x I -Q + \sum_{i=1}^n z_i \tau_i \mathbf{u}_i \mathbf{u}_i^{\T} \Big)  \\
&\times  \det\Big(x I + Q + \sum_{i=1}^n z_i \tau_i \mathbf{u}_i \mathbf{u}_i^{\T} \Big),
\end{align*}
and this is a real rooted polynomial in $x$.
\end{proposition}

\begin{proof}
For each $i$, let $\beta_i$ denote the maximum value of $\abs{\xi_i}$
among outcomes in the (finite) support of $\xi_i$.
We begin by restricting the domain of our polynomials to $x>
\|Q\|+2\sum_{i=1}^n \beta_i \|\mathbf{u}_i \mathbf{u}_i^{\T} \|$.  We then proceed by induction. 
Our induction hypothesis will be that for $0 \leq k \leq n$
\begin{align} \label{eq:inductionhyp}
& ~ \E \left[ \det\left(x^2 I - \Big(Q+ \sum_{i=1}^n (\xi_i-\mu_i) \mathbf{u}_i \mathbf{u}_i^{\T} \Big)^2
  \right) \right] \notag \\
= & ~  
\E_{\xi_{k+1}, \dots, \xi_{n}} \prod_{i=1}^k \left(1 - \frac{\partial_{z_i}^2}{2} \right) \Bigg|_{z_i = 0} \det\Bigg(x I -  Q - \sum_{i=k+1}^n (\xi_i-\mu_i) \mathbf{u}_i \mathbf{u}_i^{\T} + \sum_{j=1}^{k} z_j \tau_j \mathbf{u}_j  \mathbf{u}_j^{\T} \Bigg) \nonumber \\
&\qquad \qquad \qquad \qquad \qquad ~~~~~ \times \det\Bigg(x I +  Q+ \sum_{i=k+1}^n (\xi_i - \mu_i) \mathbf{u}_i \mathbf{u}_i^{\T} + \sum_{j=1}^k z_j \tau_j \mathbf{u}_j \mathbf{u}_j^{\T} \Bigg)
\end{align}
The base case, $k=0$ is trivially true as we get the same formula on
both sides after recalling that for any matrix $Y$
$$
\det\left(x^2 I - Y^2 \right) = \det( x I - Y) \det( x I + Y).
$$

By our assumption that $x> \|Q\|+2\sum_{i=1}^n \beta_i \|\mathbf{u}_i \mathbf{u}_i^{\T} \|$, 
\begin{align*}
x I -  \sum_{i=k+2}^n (\xi_i - \mu_i) \mathbf{u}_i \mathbf{u}_i^{\T}  + \sum_{j=1}^{k} z_j \tau_j \mathbf{u}_j  \mathbf{u}_j^{\T} 
\end{align*}
 is PSD for any realization of $\xi_{k+2}, \dots, \xi_{n}$ and in a neighborhood of zero for each $z_j$. 
Applying Corollary \ref{cor:univariateidentity} to the right hand side of (\ref{eq:inductionhyp}), yields
\begin{align*} 
& ~ \E \left[ \det\left(x^2 I - \Big(Q+ \sum_{i=1}^n (\xi_i-\mu_i) \mathbf{u}_i \mathbf{u}_i^{\T} \Big)^2 \right) \right]  \\
= & ~  
\E_{\xi_{k+2}, \dots, \xi_{n}} \prod_{i=1}^{k+1} \left(1 - \frac{\partial_{z_i}^2}{2} \right) \Bigg|_{z_i = 0} \det\Bigg(x I -  Q -  \sum_{i=k+2}^n (\xi_i-\mu_i) \mathbf{u}_i \mathbf{u}_i^{\T} + \sum_{j=1}^{k+1} z_j \tau_j \mathbf{u}_j  \mathbf{u}_j^{\T} \Bigg) \\
& \qquad \qquad \qquad \qquad \qquad ~~~~~ \times \det\Bigg(x I +  Q +  \sum_{i=k+2}^n (\xi_i-\mu_i) \mathbf{u}_i \mathbf{u}_i^{\T} + \sum_{j=1}^{k+1} z_j \tau_j \mathbf{u}_j  \mathbf{u}_j^{\T} \Bigg)
\end{align*}
which completes the induction.  To extend the proof to all $x$, we
remark that we have shown the equivalence of two polynomials in the
interval $x> \|Q\|+2\sum_{i=1}^n \beta_i \|\mathbf{u}_i \mathbf{u}_i^{\T} \|$, and two
polynomials that agree on an interval are identical. 

Real-rootedness of the right hand side follows by
Lemma~\ref{lem:determinantstable}, Lemma~\ref{lem:derivativestable},
and that by Proposition~\ref{prop:setreal} restriction to $z = {\bf 0}$ preserves
real-stability, and a univariate real stable polynomial is real rooted.
\end{proof}

\section{Defining the Interlacing Family}
The next proposition establishes that it suffices to bound the largest root of a single member of our interlacing family.

\begin{proposition} \label{prop:interlacing}
There exists a choice of outcomes  $\eps_1,\ldots,\eps_n$
in the finite support of $\xi_1, \ldots, \xi_n$, s.t.
$$
\left \| \sum_{i=1}^n \eps_i \mathbf{u}_i \mathbf{u}_i^{\T} - \sum_{i=1}^n \mu_i \mathbf{u}_i \mathbf{u}_i^{\T} \right \| 
$$
is less than the largest root of 
$$
\E_{ \xi_1, \cdots, \xi_n } \left[ \det\left(x^2 I - \Big( \sum_{i=1}^n (\xi_i - \mu_i) \mathbf{u}_i \mathbf{u}_i^{\T} \Big)^2 \right) \right].
$$
where $\E[\xi_i] = \mu_i, \forall i \in [n]$.
\end{proposition}
\begin{proof}

For a vector of independent random variables $(\xi_1, \dots, \xi_n)$
with finite support,  
let $p_{i, x}$ be the probability that $\xi_i = x$. For $\mathbf{s} = (\eps_1,
\dots, \eps_n)$ in the support of $\xi_1, \ldots, \xi_n$, we define 
\begin{align*}
q_{\mathbf{s}}(x) := \prod_{i=1}^n p_{i, \eps_i} \det\left(x^2 I - \Big(\sum_{i=1}^n (\eps_i - \mu_i) \mathbf{u}_i \mathbf{u}_i^{\T}\Big)^2 \right) .
\end{align*}  
Let $\mathbf{t}$ be a vector of $k$ outcomes in the support of  $\xi_1, \ldots,
\xi_k$, i.e. a partial assignment of outcomes.
Then we consider the conditional expected polynomial
$$
q_{\mathbf{t}}(x) := \left( \prod_{i=1}^k p_{i, t_i} \right) \E_{\xi_{k+1}, \dots, \xi_{n}} \left[ \det \left( x^2 I - \left( \sum_{i=1}^k (t_i-\mu_i) \mathbf{u}_i \mathbf{u}_i^{\T} + \sum_{j=k+1}^n  (\xi_j - \mu_j) \mathbf{u}_j \mathbf{u}_j^{\T} \right)^2 \right) \right] 
$$
which coincides with (\ref{eq:conditionalpoly}).
We show that $q_s$ is an interlacing family.  
Let $r_{k+1}^{(1)}\ldots r_{k+1}^{(l)}$ be the outcomes in the support of $\xi_{k+1}$.
For a given $\mathbf{t}$ and $r \in \setof{r_{k+1}^{(1)}, \ldots, r_{k+1}^{(l)}}$, 
let $(\mathbf{t},r)$ denote the vector $(t_1, \dots, t_k, r)$.
By Lemma \ref{lem:commoninterlacing}, it suffices to show that for any
choice of non-negative numbers $\alpha_1, \ldots, \alpha_{l}$ such
that $\sum_j \alpha_j = 1$,
the polynomial
$$
\sum_j \alpha_j q_{\mathbf{t}, r_{k+1}^{(j)}} (x)
$$
is real rooted.
 We can consider these $\alpha$'s as a probability distribution and
 define $p_{k+1, r_{k+1}^{(j)} } = \alpha_j$.
 Therefore,
\begin{align*}
q_{\mathbf{t}}(x) &= \left( \prod_{i=1}^k p_{i, t_i} \right) \E_{\xi_{k+1}, \dots, \xi_{n}} \left[ \det \left( x^2 I - \left( \sum_{i=1}^k (t_i-\mu_i) \mathbf{u}_i \mathbf{u}_i^{\T} + \sum_{j=k+1}^n  (\xi_j-\mu_j) \mathbf{u}_j \mathbf{u}_j^{\T} \right)^2 \right) \right]  \\
&= \sum_{j} \left( \prod_{i=1}^k p_{i, t_i} \right) p_{k+1, r_{k+1}^{(j)}} \\
&~~~~~~ \times \E_{\xi_{k+2}, \dots, \xi_{n}} \left [\det\left( x^2 I
  - \Bigg(\sum_{i=1}^k (t_i-\mu_i) \mathbf{u}_i \mathbf{u}_i^{\T} - (r_{k+1}^{(j)} - \mu_{k+1}) \mathbf{u}_{k+1} \mathbf{u}_{k+1}^{\T} - \sum_{j=k+2}^n (\xi_j - \mu_j)\mathbf{u}_j \mathbf{u}_j^{\T} \Bigg)^2\right) \right] \\
&= \sum_j \alpha_j q_{\mathbf{t}, r_{k+1}^{(j)}} (x). 
\end{align*}  

  By Proposition \ref{prop:expectedcharpoly}, $q_{\mathbf{t}}$ is real rooted,
  which completes the proof that $q_{\mathbf{s}}$ is an interlacing family.
  Finally, by Lemma \ref{lem:rootbound}, there exists a choice of $\mathbf{t}$
  so that the largest root of $q_{\mathbf{t}}$ is upperbounded by $q_\emptyset$.

\end{proof}

\section{Largest Root of the Expected Characteristic Polynomial}
We use the barrier method approach to control the largest root \cite{MSS15b}.
\begin{definition}
For a multivariate polynomial $p(z_1, \dots, z_n)$, we say $\mathbf{z} \in \R^n$ is \emph{above} all the roots of $p$ if for all $\mathbf{t} \in \R_{+}^n$, 
$$
p(\mathbf{z}+\mathbf{t}) > 0.
$$
We use $\Ab_p$ to denote the set of points that are above all the roots of $p$.
\end{definition}
We use the same barrier function as in \cite{bss12, MSS15b}.
\begin{definition}
For a real stable polynomial $p$ and $\mathbf{z} \in \Ab_p$, the barrier function of $p$ in direction $i$ at $\mathbf{z}$ is
$$
\Phi_p^i(\mathbf{z}) := \frac{\partial_{z_i} p(\mathbf{z})}{p(\mathbf{z})} .
$$
\end{definition}

We will also make use of the following lemma that controls the
deviation of the roots after applying a second order differential
operator.
The lemma is a slight variation of Lemma 4.8 in \cite{ag14}.
\begin{lemma} \label{lem:barrierupdate} 
Suppose that $p$ is real stable and $\mathbf{z} \in \Ab_p$. 

If 
$
\Phi_p^j(\mathbf{z}) < \sqrt{2},
$
then
$\mathbf{z} \in \Ab_{(1-
  \frac{1}{2} \partial_{z_j}^2) p}$.
If additionally for $\delta > 0$,
$$
\frac{1}{\delta} \Phi_p^j(\mathbf{z}) + \frac{1}{2} \Phi_p^j(\mathbf{z})^2 \leq 1,
$$
then 
 for all $i$,
$$
\Phi_{(1-\frac{1}{2} \partial_{z_j}^2) p}^i (\mathbf{z} + \delta \cdot \mathbf{1}_j) \leq \Phi_p^i(\mathbf{z}).
$$
\end{lemma}
We provide a proof in the Appendix~\ref{sec:omitted_proof} for completeness.

We can now bound the largest root of the expected characteristic polynomial for subisotropic vectors.  
\begin{proposition} \label{prop:subisotropic}
Let $\xi_1, \dots, \xi_n$ be independent scalar random variables
with finite support, with $\E [\xi_i] = \mu_i$, and let $\tau_i^2 = \E[ (\xi_i -\mu_i)^2 ]$.

 Let $\mathbf{v}_1, \dots, \mathbf{v}_n \in
\R^n$ such that $\sum_{i=1}^n \tau_i^2  (\mathbf{v}_i \mathbf{v}_i^{\T})^2 \preceq I.$
 Then the largest root of 
$$
p(x) := \E_{\xi_1,\cdots,\xi_n} \left [ \det \left( x^2 I - \Big(
    \sum_{i=1}^n (\xi_i  -\mu_i ) \mathbf{v}_i \mathbf{v}_i^{\T} \Big)^2 \right) \right]
$$
is at most $4$.
\end{proposition}
\begin{proof}
Note that we must have 
\begin{equation}
  \label{eq:onetermnorm}
  \max_{i \in [n] } \tau_i  \mathbf{v}_i^{\T} \mathbf{v}_i \leq  1,
\end{equation}
since otherwise $\sum_{i=1}^n \tau_i^2  (\mathbf{v}_i \mathbf{v}_i^{\T})^2 \preceq I$ is false.



Let 
$$
Q(x,\mathbf{z}) = \left( \det\Big(x I + \sum_{i=1}^n z_i \tau_i \mathbf{v}_i \mathbf{v}_i^{\T}\Big) \right)^2.
$$

Let $0 < t < \alpha$ be constants to be determined later.  Define $\delta_i = t \tau_i \mathbf{v}_i^{\T} \mathbf{v}_i$.  We evaluate our polynomial to find that 
\begin{align*}
Q(\alpha, -\delta_1, \dots, -\delta_n) &= \left( \det\Big(\alpha I - \sum_{i=1}^n \delta_i \tau_i  \mathbf{v}_i \mathbf{v}_i^{\T}\Big) \right)^2 \\
&= \left( \det\Big(\alpha I - t \sum_{i=1}^n \tau_i^2 (\mathbf{v}_i^{\T} \mathbf{v}_i)  \mathbf{v}_i \mathbf{v}_i^{\T}\Big) \right)^2 \\
&= \left( \det\Big(\alpha I - t \sum_{i=1}^n \tau_i^2 (\mathbf{v}_i \mathbf{v}_i^{\T})  \mathbf{v}_i \mathbf{v}_i^{\T}\Big) \right)^2 \\
&\geq \left( \det\Big((\alpha - t) I \Big)\right)^2\\
& > 0
  .
\end{align*}
where the second to last step follows by $ \sum_{i=1}^n \tau_i^2 (\mathbf{v}_i \mathbf{v}_i^{\T})  \mathbf{v}_i \mathbf{v}_i^{\T} \preceq I $.

This implies that $(\alpha, -\boldsymbol{\delta}) \in \R^{n+1}$ is above the roots of $Q(x,\mathbf{z})$.
We can upper bound $\Phi^i_Q(\alpha, -\boldsymbol{\delta})$ via Fact \ref{fact:Jacobi} as follows
\begin{align*}
\Phi_Q^i (\alpha, -\boldsymbol{\delta}) &= \frac{\partial_{z_i} Q}{Q} \Bigg|_{x = \alpha, \mathbf{z} = -\mathbf{\delta}} \\
&= \frac{2 \det\Big(x I + \sum_{i=1}^n z_i \tau_i \mathbf{v}_i \mathbf{v}_i^{\T}\Big) \partial_{z_i} \det\Big(x I + \sum_{i=1}^n z_i \tau_i \mathbf{v}_i \mathbf{v}_i^{\T}\Big)}{ \left( \det\Big(x I + \sum_{i=1}^n z_i \tau_i \mathbf{v}_i \mathbf{v}_i^{\T}\Big) \right)^2} \Bigg|_{x = \alpha, \mathbf{z} = -\boldsymbol{\delta}}\\
&= 2\Tr\left[ \Big(x I + \sum_{i=1}^n z_i \tau_i \mathbf{v}_i \mathbf{v}_i^{\T}\Big)^{-1} \tau_i \mathbf{v}_i \mathbf{v}_i^{\T} \right] \Bigg|_{x = \alpha, \mathbf{z} = -\boldsymbol{\delta}}\\
&= 2\Tr\left[ \Big(\alpha I - t \sum_{i=1}^n \tau_i^2 (\mathbf{v}_i \mathbf{v}_i^{\T})^2\Big)^{-1} \tau_i \mathbf{v}_i \mathbf{v}_i^{\T} \right] \\
&\leq 2\Tr\left[ \Big(\alpha - t \Big)^{-1} \tau_i \mathbf{v}_i \mathbf{v}_i^{\T} \right] \\
&= \frac{2 \tau_i \mathbf{v}_i^{\T} \mathbf{v}_i}{\alpha - t}
,
\end{align*}
where the fifth step follows by 
$ \sum_{i=1}^n \tau_i^2 (\mathbf{v}_i \mathbf{v}_i^{\T})  \mathbf{v}_i \mathbf{v}_i^{\T} \preceq I $.

Choosing $\alpha = 2t$ and $t = 2$, and recalling
Condition~\eqref{eq:onetermnorm} we get
\begin{align*}
\Phi_Q^i (\alpha, -\boldsymbol{\delta}) \leq \frac{2 \tau_i \mathbf{v}_i^{\T} \mathbf{v}_i}{\alpha -
  t}
\leq 1 < \sqrt{2}.
\end{align*}
By Lemma~\ref{lem:derivativestable}, 
$(1- \frac{1}{2} \partial_{z_i}^2) Q$ is real stable.
By Lemma \ref{lem:barrierupdate}, and since
$(4, -\boldsymbol{\delta} )\in \Ab_{(1-
  \frac{1}{2} \partial_{z_i}^2) Q}$, we have 
$(4, -\boldsymbol{\delta} + \delta_i \mathbf{1}_i)\in \Ab_{(1-
  \frac{1}{2} \partial_{z_i}^2) Q}$.
Also
\begin{align*}
\frac{1}{\delta_i} \Phi_Q^i(\alpha, -\boldsymbol{\delta}) + \frac{1}{2}
  \Phi_Q^i(\alpha, -\boldsymbol{\delta})^2
  &\leq \frac{1}{t \tau_i \mathbf{v}_i^{\T} \mathbf{v}_i} \frac{2 \tau_i \mathbf{v}_i^{\T}
    \mathbf{v}_i}{t} + \frac{1}{2} \left(\frac{2 \tau_i \mathbf{v}_i^{\T} \mathbf{v}_i}{t} \right)^2 \\
&\leq \frac{4}{t^2} \\
&= 1. 
\end{align*}
Therefore, by Lemma \ref{lem:barrierupdate}, for all $j$
$$
\Phi_{(1- \frac{1}{2} \partial_{z_i}^2) Q}^j (4, \mathbf{z} + \delta_i \mathbf{1}_i) \leq \Phi_Q^j(4,\mathbf{z}).
$$
Repeating this argument for each $i \in [n]$, demonstrates that $(4, -\boldsymbol{\delta} + \sum_{i=1}^n \delta_i \mathbf{1}_i) = (4, 0, \dots, 0)$ lies above the roots of 
$$
\prod_{i=1}^n \left(1 - \frac{ \partial_{z_i}^2}{2} \right) \left(
  \det\Big(x I  + \sum_{i=1}^n z_i \tau_i \mathbf{u}_i \mathbf{u}_i^{\T} \Big) \right)
\left( \det\Big(x I + \sum_{i=1}^n z_i \tau_i \mathbf{u}_i \mathbf{u}_i^{\T} \Big)
\right)
.
$$
After restricting to $z_i = 0$ for all $i$, we then conclude by Proposition \ref{prop:expectedcharpoly} that this is equivalent to a bound on the largest root of the expected characteristic polynomial.
\end{proof}

Having developed the necessary machinery, we now prove our main theorem.
\begin{proof} [Proof of Theorem \ref{thm:mainexpct}]
Define $\mathbf{v}_i = \frac{\mathbf{u}_i}{\sqrt{\sigma}}$.  
Then, 
$\| \sum_{i=1}^n \tau_i^2 (\mathbf{v}_i \mathbf{v}_i^{\T})^2 \| = 1$.  
Applying Proposition \ref{prop:subisotropic} and Proposition~\ref{prop:interlacing}, we conclude that there exists a choice of outcomes $\eps_i$ such that
$$
\left\|  \sum_{i=1}^n (\eps_i  -\mu_i ) \mathbf{v}_i \mathbf{v}_i^{\T} \right \| \leq 4.  
$$
From this, we conclude that
$$
\left\| \sum_{i=1}^n  (\eps_i  -\mu_i ) \mathbf{u}_i \mathbf{u}_i^{\T} \right \| \leq 4 \sigma.  
$$
\end{proof}

\newpage
\addcontentsline{toc}{section}{References}
\bibliographystyle{plain}
\bibliography{ref_2,ref}
\newpage
\appendix

\section{Omitted Proofs}\label{sec:omitted_proof}

\subsection{Proof of Lemma~\ref{lem:barrierupdate}}

Choosing $c=1/2$ in Lemma~\ref{lem:barrierupdate_general} gives a
proof of Lemma~\ref{lem:barrierupdate}.
\begin{lemma}[Generalization of Lemma 4.8 in \cite{ag14}]\label{lem:barrierupdate_general}
Suppose that $p(z_1, \cdots, z_m)$ is real stable and $z \in
Ab_p$. For any $c \in [0,1]$. 

If 
$
\Phi_p^j(\mathbf{z}) < \sqrt{1/c},
$
then
$\mathbf{z} \in \Ab_{(1-
  c\partial_{z_j}^2) p}$.
If additionally for $\delta > 0$,
\begin{align*}
c \cdot \left( \frac{2}{\delta} \Phi_p^j (\mathbf{z}) +  ( \Phi_p^j (\mathbf{z}) )^2 \right) \leq 1,
\end{align*}
then, for all $i \in [m]$,
\begin{align*}
\Phi_{(1- c\partial_j^2) p }^i ( \mathbf{z} + \delta
{\bf 1}_j ) \leq \Phi_p^i (
\mathbf{z}).
\end{align*}
\end{lemma}
\begin{proof}
We write $\partial_i$ instead of $\partial_{z_i}$ for ease of
notation. 
By Lemma~\ref{lem:derivativestable}, 
$(1- c\partial_j^2) p$ is real stable.
Recall the definitions of $\Phi_p^j(\mathbf{z})$ and $\Psi_p^j(\mathbf{z})$ 
\begin{align*}
\Phi_p^j(\mathbf{z}) = \frac{ \partial_{z_j} \det(M) }{ \det(M) } =
  \frac{\partial_{j} p}{p} \text{ and } \Psi_p^j(\mathbf{z}) = \frac{ \partial_{j}^2 \det(M) }{ \det(M) } = \frac{\partial_j^2 p}{ p }.
\end{align*}
Consider a non-negative vector $\mathbf{t}$. By Lemmas 4.6 and 4.5 in \cite{ag14}, we have
\begin{align}
\label{eq:basicbounds}
\Phi_p^j (\mathbf{z} + \mathbf{t}) \leq \Phi_p^j(\mathbf{z}), ~~~~~~ \Psi_p^j(\mathbf{z} + \mathbf{t} ) \leq \Phi_p^j(\mathbf{z}+\mathbf{t} )^2 \leq \Phi_p^j(\mathbf{z})^2.
\end{align}
Thus, $\Psi_p^j(\mathbf{z} + \mathbf{t} ) \leq \Phi_p^j(\mathbf{z})^2 < 1/c$
so $c \cdot \partial_j^2  p(\mathbf{z} + \mathbf{t}) <  p(\mathbf{z} + \mathbf{t})$, i.e.
$(1-c\partial_j^2)  p(\mathbf{z} + \mathbf{t}) >0$.
Thus $\mathbf{z} \in \Ab_{(1-c\partial_{z_j}^2) p}$.

Next, we write $\Phi_{p - c \cdot \partial_j^2 p}^i$ in terms of $\Phi_p^i$ and $\Psi_p^j$ and $\partial_i \Psi_p^j$.
\begin{align*}
\Phi_{ p - c \cdot \partial_j^2 p }^i = & ~ \frac{ \partial_i ( p - c \cdot \partial_j^2 p ) }{ p - c \cdot \partial_j^2 p } \\ 
= & ~ \frac{ \partial_i ( ( 1 - c \cdot \Psi_p^j ) p ) }{ (1 - c \cdot \Psi_p^j ) p } \\
= & ~ \frac{ ( 1 - c \cdot \Psi_p^j ) ( \partial_i p ) }{ ( 1 - c \cdot \Psi_p^j ) p } + \frac{ ( \partial_i ( 1 - c \cdot \Psi_p^j ) ) p }{ ( 1 - c \cdot \Psi_p^j ) p } \\
= & ~ \Phi_p^i - \frac{ c \cdot \partial_i \Psi_p^j }{ 1 - c \cdot \Psi_p^j }.
\end{align*}
We would like to show that $\Phi_{ p - c \cdot \partial_j^2 p }^i ( \mathbf{z} + \delta {\bf 1}_j ) \leq \Phi_p^i (\mathbf{z})$. Equivalently, it is enough to show that
\begin{align*}
- \frac{ c \cdot \partial_i \Psi_p^j (\mathbf{z} + \delta {\bf 1}_j ) }{ 1 - c \cdot \Psi_p^j ( \mathbf{z} + \delta {\bf 1}_j ) } \leq \Phi_p^i (\mathbf{z}) - \Phi_p^i (\mathbf{z} + \delta {\bf 1}_j ).
\end{align*}
By convexity, it is enough to show that
\begin{align*}
- \frac{ c \cdot \partial_i \Psi_p^j ( \mathbf{z} + \delta {\bf 1}_j ) }{ 1 - c \cdot \Psi_p^j ( \mathbf{z} + \delta {\bf 1}_j ) } \leq \delta \cdot ( - \partial_j \Phi_p^i ( \mathbf{z} + \delta {\bf 1}_j ) ).
\end{align*}
By monotonicity, $\delta \cdot ( - \partial_j \Phi_p^i (\mathbf{z} + \delta
{\bf 1}_j) ) > 0$ so we may divide both sides of the above inequality
by this term and obtain that the above is equivalent to 
\begin{align*}
\frac{ - c \cdot \partial_i \Psi_p^j ( \mathbf{z} + \delta {\bf 1}_j ) }{ - \delta \cdot \partial_i \Phi_p^j (\mathbf{z} + \delta {\bf 1}_j ) } \cdot \frac{1}{ 1 - c \Psi_p^j (\mathbf{z} + \delta {\bf 1}_j ) } \leq 1,
\end{align*}
where we also used $\partial_j \Phi_p^i = \partial_i \Phi_p^j$. By Lemma 4.10 in \cite{ag14}, $ \frac{ \partial_i \Psi_p^j }{ \partial_i \Phi_p^j }  \leq 2 \Phi_p^j $. So, we can write,
\begin{align*}
\frac{2 c}{ \delta } \Phi_p^j (\mathbf{z} + \delta {\bf 1}_j ) \cdot \frac{1}{ 1 - c \cdot \Psi_p^j ( \mathbf{z} + \delta {\bf 1}_j ) } \leq 1.
\end{align*}
By Equation~\eqref{eq:basicbounds}, with $\mathbf{t} = \delta {\bf 1}_j$,
\begin{align*}
\Phi_p^j (\mathbf{z} + \delta {\bf 1}_j) \leq \Phi_p^j(\mathbf{z}), ~~~~~~ \Psi_p^j(\mathbf{z} + \delta {\bf 1}_j ) \leq \Phi_p^j(\mathbf{z}+ \delta {\bf 1}_j )^2 \leq \Phi_p^j(\mathbf{z})^2.
\end{align*}
So, it is enough to show that
\begin{align*}
\frac{ 2 c }{ \delta } \Phi_p^j(\mathbf{z}) \cdot \frac{1}{ 1 - c \cdot \Phi_p^j(\mathbf{z})^2 } \leq 1.
\end{align*}
Using $\Phi_p^j(\mathbf{z}) < 1$ and $ c \in [0,1]$ we know that $1 - c \cdot
\Phi_p^j(\mathbf{z})^2 > 0$. We can multiply both sides with $1 - c \cdot
\Phi_p^j(\mathbf{z})^2$ and we obtain that it suffices to have
\begin{align*}
c \cdot \left( \frac{2}{\delta} \Phi_p^j(\mathbf{z}) + \Phi_p^j(\mathbf{z})^2 \right) \leq 1,
\end{align*}
which is true by assumption.
\end{proof}



\end{document}